\documentclass[final,1p,times]{elsarticle}

\usepackage[utf8]{inputenc}
\usepackage{amsmath}
\usepackage{amssymb}
\usepackage{amsthm}
\usepackage{csquotes}
\usepackage{color}
\usepackage[shortlabels]{enumitem}

\theoremstyle{definition}
\newtheorem{theorem}{Theorem}

\newtheorem{prop}{Proposition}
\newtheorem{corollary}{Corollary}

\numberwithin{equation}{section}
\usepackage{hyperref}

\journal{Journal of Algebra}

\begin{document}

\begin{frontmatter}



\title{Polynomials over Ring of Integers of Global Fields that have Roots Modulo Every Finite Indexed Subgroup}


\author[1]{Bhawesh Mishra}

\affiliation[1]{organization={Department of Mathematics, The Ohio State University},
            addressline={$231$ W $18^{th}$ Ave}, 
            city={Columbus},
            postcode={43210}, 
            state={OH},
            country={USA}}

\begin{abstract}
A polynomial with coefficients in the ring of integers $\mathcal{O}_{K}$ of a global field $K$ is called intersective if it has a root modulo every finite-indexed subgroup of $\mathcal{O}_{K}$. We prove two criteria for a polynomial $f(x)\in\mathcal{O}_{K}[x]$ to be intersective. One of these criteria is in terms of the Galois group of the splitting field of the polynomial, whereas the second criterion is verifiable entirely in terms of constants which depend upon $K$ and the polynomial $f$. The proofs use the theory of global field extensions and upper bound on the least prime ideal in the Chebotarev density theorem.  
\end{abstract}

\end{frontmatter}


\section{Introduction}
\subsection{Motivation}
Let $q = p^{r}$ be power of a prime $p$ and $T$ be transcendental over $\mathbb{F}_{q}$. Let $\mathbb{F}_{q}(T)$ denote the field of rational functions with constant field $\mathbb{F}_{q}$ and $\mathbb{F}_{q}[T]$ denote the ring of polynomials with coefficients in $\mathbb{F}_{q}$. Recall that an extension $K$ of a field $M$ is said to be separable if for every $\alpha\in K$, the minimal polynomial $p_{\alpha}$ has distinct roots in the algebraic closure. For the ease of discussion, we will say that $K$ is a global field when $K$ is either of the following:
\begin{itemize}
    \item An algebraic number field. 
    
    \item A finite and separable extension of $\mathbb{F}_{q}(T)$, where $p$ is a prime and $q = p^{r}$ for some $r \geq 1$. 
\end{itemize}
Let $\mathcal{O}_{K}$ be the ring of integers of $K$ when $K$ is a number field, whereas let $\mathcal{O}_{K}$ be the integral closure of the $\mathbb{F}_{q}[T]$ in $K$ in the function field case. This article is about polynomials $f(x)\in\mathcal{O}_{K}[x]$ such that for every subgroup $\Gamma$ of $\mathcal{O}_{K}$ of finite index, there exists $x \in\mathcal{O}_{K}$ with 
\begin{equation}
f(x) \equiv 0 \hspace{1mm} (\text{mod } \Gamma).\label{defn1}
\end{equation}
Obviously, any $f$ that has a root in $\mathcal{O}_{K}$ satisfies \eqref{defn1} for trivial reasons. The more interesting case is when $f$ satisfies \eqref{defn1} but does not have a root in $K$. Each such $f$ constitutes a counterexample to the local-global principle.  

Another motivation for studying polynomials $f$ that satisfy \eqref{defn1} comes from combinatorial number theory. A set $S \subset\mathcal{O}_{K}$ is said to be intersective if given any subset $T \subset\mathcal{O}_{K}$ with positive upper density, one has $(T - T) \cap S \not\subseteq \{0\}$. The upper density of $S \subset\mathcal{O}_{K}$ is given by $$\overline{d}(S) = \limsup_{N\rightarrow\infty} \frac{|S \cap \Phi_{N}|}{|\Phi_{N}|},$$ where the sequence of sets $\Phi_{N}$ are defined as :
\begin{gather*}
    \Phi_{N} = \begin{cases} \{ \alpha\in\mathcal{O}_{K} : \text{Norm}(\alpha) \leq N\} \text{ ; when $K$ is a number field }\\ \{ \alpha \in \mathcal{O}_{K} : \text{ deg}(\alpha) \leq N \} \text{ ; when $K$ is a finite separable extension of } \mathbb{F}_{q}(T) . \end{cases}
\end{gather*}
For example, when $K = \mathbb{Q}$, the sequence $\Phi_{N}$ can be taken to be  $[-N, N] \cap {Z}$ for every $N \geq 1$. When $K$ is a finite separable extension of $\mathbb{F}_{q}(T)$, the sequence $\Phi_{N}$ is the set of divisors of degree $\leq N$. 

A polynomial $f(x) \in\mathcal{O}_{K}[x]$ is said to be intersective if $\{f(\alpha) : \alpha\in\mathcal{O}_{K}\}$ is an intersective set. It was shown in \cite{KaMF} that a polynomial $f(x) \in\mathbb{Z}[x]$ is intersective if and only if $f$ satisfies \eqref{defn1}. A far reaching extension of this result was obtained in \cite{BLL}, a special instance of which implies that for a polynomial $f(x) \in\mathbb{Z}[x]$ every subset of positive density in $\mathbb{Z}$ contains arbitrarily long polynomial progressions of the form $$a, a+f(n), a+2f(n), \ldots, a+kf(n)$$ if and only if $f(x)$ is an intersective polynomial. A polynomial $f(x)$ in $\mathcal{O}_{K}[x]$ is intersective if and only if it satisfies $\eqref{defn1}$. This is implied by a very special case of Theorem $1.6$ in \cite{BerRob} when $K$ is a number field, and Theorem $1.15$ in \cite{BA} gives the analogous result when $K$ is a finite separable extension of $\mathbb{F}_{q}(T)$. The same results in \cite{BA} and \cite{BerRob} also imply that intersectivity of polynomial $f(x) \in\mathcal{O}_{K}[x]$ is equivalent to the fact that for every $\alpha\in\mathcal{O}_{K}$, there exists $\beta\in\mathcal{O}_{K}$ such that $f(\beta) \equiv 0 \hspace{1mm} (\text{mod } \alpha)$, i.e., $f(\beta)$ is in the ideal generated by $\alpha$. 

Berend and Bilu, in \cite{BerBil}, proved a criterion to decide whether a polynomial $f(x) \in\mathbb{Z}[x]$ is intersective. In this article, we generalize the results in \cite{BerBil} to polynomials $f(x) \in\mathcal{O}_{K}[x]$ when $K$ is a global field. The methods in this article are generalizations of those in \cite{BerBil} in the settings of number fields and function fields respectively. In the process of proving the main result, we also derive an effective upper bound on the least prime ideal in the Chebotarev density theorem for function fields (see Proposition \ref{leastprime2}).

\subsection{Global Field Extensions and the Artin Symbols}
Let $L$ be a Galois extension of a global field $K$. We will denote the set of primes in $\mathcal{O}_{K}$ by $\mathbb{P}(K)$. When $K$ is an extension of $\mathbb{F}_{q}(T)$, its field of constants is $\mathbb{F}_{q^{l}}$ for some $l \geq 1$. We will denote the set of primes of degree $n$ by $\mathbb{P}_{n}(K)$. Given a prime $\mathfrak{p} \in\mathbb{P}(K)$, $K_{\mathfrak{p}}$ will denote the corresponding completion of $K$ and $|\cdot|_{\mathfrak{p}}$ will denote the corresponding absolute value. If $\mathfrak{p} \in\mathbb{P}(K)$ is unramified and $\beta\in\mathbb{P}(L)$ lies over $\mathfrak{p}$ then the Frobenius symbol $\big(\frac{L/K}{\beta}\big)$ denotes the unique automorphism $\varphi_{\beta}\in$ Gal$(L/K)$ such that 
\begin{equation}
    \varphi_{\beta}(\alpha) \equiv \alpha^{N\mathfrak{p}} \hspace{1mm} (\text{mod } \beta) \label{Frobenius}
\end{equation}
for every $\alpha \in\mathcal{O}_{L}$ and $N\mathfrak{p} = [\mathcal{O}_{K}:  \mathfrak{p}]$. 

Let $L$ be the splitting field of a polynomial $f(x)\in\mathcal{O}_{K}[x]$, all of whose irreducible factors are separable, and $G$ be the Galois group of this extension $L/K$. Note that when the characteristic of $K$ is $p > 0$, an irreducible polynomial $f(x) \in\mathcal{O}_{K}[x]$ is not separable if and only if $f(x) = g(x^{p^{n}})$ for some polynomial $g$ and some integer $n \geq 1$. Any prime $\mathfrak{p}$ in $\mathcal{O}_{K}$ that does not divide the discriminant of $f$, is unramified and hence splits as $\mathfrak{p} = \beta_{1} \cdot \beta_{2} \cdots \beta_{l}$ for some primes $\beta_{1}, \beta_{2}, \ldots, \beta_{l}$ in $L$. We use the Artin symbol $\big[ \frac{L/K}{\mathfrak{p}} \big]$ to denote the set of conjugacy class of automorphisms
$$\bigg[ \frac{L/K}{\mathfrak{p}} \bigg] = \bigg\{ \bigg( \frac{L/K}{\beta_{1}} \bigg), \bigg(\frac{L/K}{\beta_{2}} \bigg), \ldots, \bigg( \frac{L/K}{\beta_{l}} \bigg) \bigg\}.$$ Let $M$ be a subfield of $L$ and $\mathfrak{p}^{\prime}$ be a prime in $M$ that lies above $\mathfrak{p}$. Then, for any $1 \leq i \leq l$ we have $$\bigg( \frac{L/M}{\beta_{i}} \bigg) = \bigg( \frac{L/K}{\beta_{i}} \bigg)^{f_{\mathfrak{p}}},$$ where $f_{\mathfrak{p}} = [M_{\mathfrak{p}^{\prime}} : K_{\mathfrak{p}}]$. We also have the natural inclusion of the groups $$\text{Gal }(L_{\beta_{i}}/M_{\mathfrak{p}^{\prime}}) \leq \text{Gal }(L/M) \leq \text{Gal }(L/K),$$ which gives that
\begin{equation}
    \bigg( \frac{L/K}{\beta_{i}}\bigg)^{m} \in \text{Gal } (L/M) \text{ if and only if } f_{\mathfrak{p}} \mid m. \label{DG}
\end{equation}
Readers may consult \cite[Section 6.2]{FrJar} for more details on Artin Symbols for global field extensions. When $L$ and $K$ are number fields, $d_{L}$ will denote the absolute discriminant of $L$ and $d_{L/K}$ will denote the relative discriminant of the extension $L/K$. Furthermore, $N_{L/K}$ will denote the corresponding norm. Given fields $M_{1}, M_{2}$, we denote by $M_{1}M_{2}$ the compositum of $M_{1}$ and $M_{2}$. The genus of a separable and finite extension $M$ of $\mathbb{F}_{q}(T)$ is denoted by $g_{M}$. The resultant of a polynomial $f(x) \in\mathcal{O}_{K}[x]$ is the resultant of $f$ and its derivative $f^{\prime}$.

\subsection{Setup and the Main Result}
Let $f(x) \in\mathcal{O}_{K}[x]$ such that no non-unit divides all the coefficients of $f$. If some non-unit $\beta$ does divide every coefficient of $f(x)$, we can replace $f(x)$ by $\frac{1}{\beta} f(x)$. Let $g_{1}(x), g_{2}(x), \ldots, g_{m}(x)$ be the irreducible factors of $f(x)$. When $K$ is an extension of $\mathbb{F}_{q}(T)$, we further assume that each of the irreducible factors of $f$ are separable polynomials.

Let $L$ be the splitting field of $f$ over $K$ and $G =$ Gal $(L/K)$ be the Galois group of this extension. For every $i \in\{1, 2, \ldots, m\}$, let $\alpha_{i}$ be a root of $g_{i}$ that is fixed. Let $L_{i} := K(\alpha_{i})$, which is an intermediate field between $K$ and $L$, let $H_{i} :=$ Gal$(L/L_{i}) \leq G$, and finally let $U := \bigcup_{i=1}^{m} H_{i} \subseteq G$. For every $1 \leq i \leq m$, let
\begin{equation*}
    g_{i}(x) = \sum_{j=0}^{n_{i}} a_{ij} x^{j}
\end{equation*}
for some $a_{i0}, \ldots, a_{in_{i}} \in\mathcal{O}_{K}$. Also, let
\begin{itemize}
    \item $\mathfrak{q}_{i}$ be the discriminant ideal of the irreducible polynomial $g_{i}$, which is non-zero because each of the $g_{i}$ is separable. 
    
    \item $\mathfrak{r}_{i} = $ resultant of $g_{i}$ and $g_{i}^{\prime}$, 
    
    \item $\mathfrak{q} = \mathfrak{r}_{1} \cdots \mathfrak{r}_{m}$,
    
    \item $\mathfrak{q} = \mathfrak{p}_{1}^{b_{1}} \cdots \mathfrak{p}_{\nu}^{b_{\nu}}$ be the unique prime factorization of $\mathfrak{q}$ in $\mathcal{O}_{K}$ and 
    
    \item Define
        $$\Delta := \mathfrak{p}_{1}^{2b_{1} + 1} \cdots \mathfrak{p}_{\nu}^{2b_{\nu} + 1}. $$
        
    \item If $K$ is a number field, define    
        
        \[D := \prod_{i=1}^{m} \big( \mathfrak{q}_{i} \big)^{(n_{1}!\cdots n_{m}!)(1 - \frac{1}{n_{i}})}.\] 
     
     \item If $K$ is extension of $\mathbb{F}_{q}(T)$, define

        $$\Delta^{\prime} := \sum_{i=1}^{\nu} b_{i} \text{ deg}(\mathfrak{p}_{i}) \in\mathbb{Z}, $$
        
         \[D^{\prime} := \prod_{i=1}^{m} n_{i}!\] and \[\text{ and } d = [K : \mathbb{F}_{q}(T)].\]
    
\end{itemize}
Our main result gives two equivalent conditions for the polynomial $f(x)$ to be intersective.
\begin{theorem}
Let $K$ be a global field and $f(x) \in\mathcal{O}_{K}[x]$ be a polynomial such that all of its irreducible factors are separable. Then, the following three statements are equivalent:
\begin{enumerate}
    \item $f(x)$ is an intersective polynomial, i.e. $f$ satisfies \eqref{defn1}.
    
    \item  The congruence $f(x) \equiv 0 \hspace{1mm} (\text{mod } \Delta)$ is solvable in $\mathcal{O}_{K}$ and $$\bigcup_{\sigma\in G} \sigma^{-1}U\sigma = G.$$
    
    \item $f(x) \equiv 0 \hspace{1mm} (\text{mod } \Delta)$ is solvable in $\mathcal{O}_{K}$ and $$f(x) \equiv 0 \hspace{1mm} (\text{mod } \mathfrak{p}) \hspace{2mm} \text{is solvable for primes }  \mathfrak{p} \in\mathcal{O}_{K}$$ with
    
    \[\begin{cases} \text{deg }(\mathfrak{p}) \leq \lceil 2 \log_{q}   (2\Delta^{\prime} + 2d D^{\prime} + 8g_{K}D^{\prime} + 4)  \rceil \hspace{2mm}; \text{ if $K$ is extension of $\mathbb{F}_{q}(T)$ } \\ N_{K/\mathbb{Q}}(\mathfrak{p}) \leq \big(N_{K/\mathbb{Q}}(D)\big)^{12577}; \text{ if $K$ is a number field  and } L \neq \mathbb{Q}. \end{cases}\] If $L = \mathbb{Q}$, then solvability of $f(x) \equiv 0 \hspace{1mm} (\text{mod } \Delta)$ is the only condition to check. 
\end{enumerate}
\end{theorem}

Readers may note that the condition $(2)$ in the above theorem involves Galois group of the splitting field. This condition is useful in creating families of examples of polynomials satisfying \eqref{defn1}, for instance, as in Section $4$. On the other hand, the condition $(3)$ is feasible for checking intersectivity of a polynomial since it only involves constants $\Delta, D, \Delta^{\prime}$ and $D^{\prime}$, all of which can be calculated from the polynomial $f$ itself.

The rest of this article is organized into three more sections. The second section contains statements and proofs of results that are required to prove Theorem $1$. Section $3$ contains the proof of Theorem $1$. Section $4$ contains some examples, consequences and discussions pertaining to the main result. 

\section{Some Preliminary Results}
\subsection{Roots Modulo Primes and the Artin Symbol}
Proposition $1$ gives a necessary and sufficient condition for an irreducible polynomial to have roots in $K_{\mathfrak{p}}$ for primes $\mathfrak{p}$ that do not divide the resultant of the polynomial $f$. Although this proposition can be proved using the decomposition and inertia subgroups of $\text{Gal } (L/K)$, we choose to provide a more accessible and elementary proof.
\begin{prop}
Let $K$ be a global field and $g(x) \in \mathcal{O}_{K}[x]$ be a separable irreducible polynomial. Let $L$ be the splitting field of $g$ and $G$ be the Galois group of $L$ over $K$. Also, let $\mathfrak{p} \in\mathcal{O}_{K}$ be a prime that does not divide the resultant $\mathfrak{r}$ of the polynomial $g$. Then $g$ has a root modulo $\mathfrak{p}$ if and only if there exists an automorphism in $\big[ \frac{L/K}{\mathfrak{p}}\big]$ that fixes some root of $g$. \label{density1}
\end{prop}

\begin{proof}
Let $\theta_{1}, \theta_{2}, \ldots, \theta_{n}$ be the distinct roots of $g(x)$ in $L$ and $\beta \in\mathbb{P}(L)$ be a prime that lies above $\mathfrak{p}$. Assume that $g(\alpha) \equiv 0 \hspace{1mm} (\text{mod } \mathfrak{p})$ for some $\alpha \in \mathcal{O}_{K}$. We will show that the Frobenius automoprhism $\varphi_{\beta} \in \big[ \frac{L/K}{\mathfrak{p}} \big]$ fixes one of the roots of $g$. Let us write $\varphi$ for $\varphi_{\beta}$ to avoid too many subscripts.

Since $$g(\alpha) = \prod_{i=1}^{n} (\alpha - \theta_{i}) \equiv 0 \hspace{1mm} (\text{mod } \mathfrak{p}),$$ we have that $\mathfrak{p}\mathcal{O}_{K} \mid \big( \prod_{i=1}^{n} \alpha - \theta_{i} \big)\mathcal{O}_{K}$, and hence $\mathfrak{p}\mathcal{O}_{L} \mid \big( \prod_{i=1}^{n} \alpha - \theta_{i} \big)\mathcal{O}_{L}$. 

Since $\beta$ is a prime above $\mathfrak{p}$, we also have $\beta\mathcal{O}_{L} \mid \mathfrak{p}\mathcal{O}_{L}$, which along with $\mathfrak{p}\mathcal{O}_{L} \mid \big( \prod_{i=1}^{n} \alpha - \theta_{i} \big)\mathcal{O}_{L}$ gives $\beta\mathcal{O}_{L} \mid \big( \prod_{i=1}^{n} \alpha - \theta_{i} \big)\mathcal{O}_{L}$. Therefore, $\beta\mathcal{O}_{L} \mid (\alpha - \theta_{j})$ for some $j \in\{1, 2, \ldots, n\}$, i.e., \[ \theta_{j} \equiv \alpha \hspace{1mm} (\text{mod } \beta).\] The reduction of $\theta_{j}$ modulo $\beta$ is well-defined because $\beta$ is a prime above $\mathfrak{p}$, and $\mathfrak{p}$ is a prime that does not divide the resultant and hence, the leading coefficient of $g$. Since $\alpha \in\mathcal{O}_{K}$, by definition of the Frobenius automorphism we have,
\begin{equation*}
    \varphi(\alpha) = \alpha \equiv \alpha^{N\mathfrak{p}} \hspace{1mm} (\text{mod } \beta).
\end{equation*} 
Therefore, we have from \eqref{Frobenius} that, 
\begin{equation*}
    \varphi(\theta_{j}) \equiv \theta_{j}^{N\mathfrak{p}} \equiv \alpha^{N\mathfrak{p}} \equiv \alpha \equiv \theta_{j} \hspace{1mm} (\text{mod } \beta)
\end{equation*}
and hence 
\begin{equation}
    \theta_{j} \equiv \varphi^{-1}(\theta_{j}) \hspace{1mm} (\text{mod } \varphi^{-1}(\beta)). \label{T}
\end{equation}
Using $T := \varphi^{-1} \in G$, \eqref{T} becomes 
\begin{equation*}
    \theta_{j} \equiv T(\theta_{j}) \hspace{1mm} (\text{mod } T(\beta)).
\end{equation*}
Since $T$ is an automorphism, $T(\theta_{j}) = \theta_{l}$ for some $l \in\{1, 2, \ldots, n\}$ and hence we have 
\begin{equation}
    (\theta_{j} - \theta_{l}) \equiv 0  \hspace{1mm} (\text{mod } T(\beta)). \label{fixes2}
\end{equation}
Here, reduction of $(\theta_{j} - \theta_{l})$ modulo $T(\beta)$ is well-defined because $\theta_{j}$ is equivalent to $\alpha \in\mathcal{O}_{K} \subseteq \mathcal{O}_{L}$ modulo $\beta$ and hence $\theta_{l} = T(\theta_{j})$ is also equivalent to an integral element of $L$. Recall that the resultant $\mathfrak{r}$ is defined to be \[a^{2n-1} \prod_{1\leq r \neq s \leq n} (\theta_{r} - \theta_{s}),\] where $a$ is the leading coefficient of $g$. Since $\theta_{1}, \theta_{2}, \ldots, \theta_{n}$ are distinct roots of $g(x)$, if $j \neq l$ then from \eqref{fixes2} we have that
\begin{equation}
\mathfrak{r} \equiv 0 \hspace{1mm} (\text{mod } T(\beta)), \label{ideal1}
\end{equation} 
i.e., $\mathfrak{r}\mathcal{O}_{L} \subseteq T(\beta)\mathcal{O}_{L}$, which implies $\mathfrak{r}\mathcal{O}_{K} \subseteq \mathfrak{p}\mathcal{O}_{K}$ because $\beta$, and hence $T(\beta)$, is a prime above $\mathfrak{p}$. However, $\mathfrak{r}\mathcal{O}_{K} \subseteq \mathfrak{p}\mathcal{O}_{K}$ is a contradiction to the fact that $\mathfrak{p}$ is a prime in $\mathcal{O}_{K}$ that does not divide the resultant $\mathfrak{r}$ of $g$. Therefore, we must have $j = l$ and hence $T (\theta_{j}) = \varphi^{-1}_{\beta} (\theta_{j}) = \theta_{j}$  i.e. the automorphism $\varphi_{\beta}$ in $\big[ \frac{L/K}{\mathfrak{p}} \big]$ fixes the root $\theta_{j}$. 

For the other direction, assume that $\theta$ is a root of $g(x)$ which is fixed by some Frobenius automorphism $\varphi_{\beta} \in \big[ \frac{L/K}{\mathfrak{p}} \big]$, where $\beta$ is a prime in $\mathcal{O}_{L}$ dividing $\mathfrak{p}$. Therefore, we have $$\varphi_{\beta}(\theta) = \theta \equiv \theta^{N\mathfrak{p}} \hspace{1mm} (\text{mod } \beta) .$$ 

Since $\theta^{N\mathfrak{p}} \equiv \theta \hspace{1mm} (\text{mod } \beta)$ we must have that $\theta \equiv \alpha \hspace{1mm} (\text{mod } \beta)$ for some $\alpha \in\mathcal{O}_{K}$. Since $\theta$ was a root of $g$, we have $g(\alpha) \equiv g(\theta) \equiv 0 \hspace{1mm} (\text{mod } \beta)$ and hence $g(\alpha) \equiv 0 \hspace{1mm} (\text{mod } \mathfrak{p})).$ 
\end{proof}

\subsection{Relating the Absolute Discriminant and the Ideal D in the Number Field Case}
We will prove two propositions that relate the ideal $D$ defined in Theorem $1$ to the absolute discriminant $d_{L}$ of the splitting field $L$ of $f$, in the case when $K$ is a number field. The first proposition in this subsection is known but we include the proof for the sake of keeping the presentation largely self-contained. A proof of the second proposition when $K = \mathbb{Z}$ appears in \cite{BerBil}, which we adapt for number fields. 
\begin{prop}
Let $M_{1}, M_{2}, \ldots, M_{\gamma}$ be finitely many finite algebraic extensions of a number field $K$ and $M$ be the composite of these finitely many extensions. Let $d_{1}, d_{2}, \ldots, d_{\gamma}$ be the relative discriminant of $M_{1}, M_{2}, \ldots, M_{\gamma}$ respectively and $d$ be that of $M$. Then $$ d \mid \bigg( \prod_{i=1}^{\gamma} d_{i}^{[M : M_{i}]} \bigg).$$
\end{prop}

\begin{proof}
We will prove the proposition by induction on $\gamma$; let $\gamma = 2$. Let $\mathfrak{D}$, $\mathfrak{D}_{1}$, $\mathfrak{D}_{2}$ and $\mathfrak{D}^{(1)}$ respectively denote the relative different of the field extensions $M/K$, $M_{1}/K$, $M_{2}/K$ and $M/M_{1}$. Note that the different $\delta_{L/K}(\alpha)$ of an element $\alpha\in\mathcal{O}_{L}$ is defined in terms of its minimal polynomial $m(x)\in\mathcal{O}_{K}[x]$ as 
\[\delta_{L/K}(\alpha) = \begin{cases} m^{\prime}(\alpha) \text{ ; if } L = K(\alpha) \\ 0 \text{ ; if } L \neq K(\alpha)
\end{cases}
\]
We will also use the fact that the different $\mathfrak{D}_{L/K}$ of a field extension $L/K$ is the ideal generated by the set of all the element-wise differents $\{\delta_{L/K}(\alpha) : \alpha\in\mathcal{O}_{L}\}$ ( see \cite[pp. 198, Theorem (2.5)]{Neukrich}). First, we want to show that $\mathfrak{D}^{(1)} \mid \mathfrak{D}_{2}$, i.e., $\mathfrak{D}_{2} \subseteq \mathfrak{D}^{(1)}$. 

Let $\zeta \in\mathcal{O}_{M_{2}}$ and $M_{2} = K(\zeta)$ without loss of generality (since $\delta_{M_{2}/K}(\zeta) = 0$ otherwise). By definition, we have that $\delta_{M_{2}/K}(\zeta) = m^{\prime}(\zeta)$ where $m(x) \in\mathcal{O}_{K}[x]$ is the minimal polynomial of $\zeta$ over $K$. On the other hand, the minimal polynomial $p(x) \in\mathcal{O}_{M_{1}}[x]$ of $\zeta$ over the field $M_{1}$ divides $m(x)$. Therefore, $\delta_{M/M_{1}}(\zeta) =  p^{\prime}(\zeta) \mid m^{\prime}(\zeta) = \delta_{M_{2}/K} (\zeta)$ which according to definition of different gives that $\delta_{M_{2}/K}(\zeta) \in \mathfrak{D}_{M/M_{1}} = \mathfrak{D}^{(1)}$. Therefore, we have that $\mathfrak{D}^{(1)} \mid \mathfrak{D}_{2}$.

Now, using the the multiplicativity of the relative differents over the tower of field extension $M / M_{1} / K$ (for instance, see \cite[pp. 195, propostion 2]{Neukrich}) we also have $$\mathfrak{D} = \mathfrak{D}^{(1)} \mathfrak{D}_{1},$$ which together with $\mathfrak{D}^{(1)} \mid \mathfrak{D}_{2}$ implies that $\mathfrak{D} \mid \mathfrak{D}_{1}\mathfrak{D}_{2}$. Upon applying $N_{M/K}$ we have that $d = N_{M/K}(\mathfrak{D})$ divides
\begin{equation*}
    N_{M/K}(\mathfrak{D}_{1}\mathfrak{D}_{2}) = N_{M_{1}/K}\big( N_{M/M_{1}}(\mathfrak{D}_{1}) \big) \cdot N_{M_{2}/K}\big( N_{M/M_{2}}(\mathfrak{D}_{2}) \big)
\end{equation*}
Since $N_{M/M_{i}}(\mathfrak{D_{i}}) = \mathfrak{D}_{i}^{[M : M_{i}]}$, we have that $N_{M_{i}/K}\big( N_{M/M_{i}}(\mathfrak{D_{i}}) \big) = d_{i}^{[M : M_{i}]}$, which proves the proposition for $\gamma = 2$. 

In the general case, given any $\gamma \geq 3$, let $\widetilde{M}$ be the composite of $M_{1}, \ldots, M_{\gamma - 1}$. Since $M$ is the composite of $\widetilde{M}$ and $M_{\gamma}$, we have $$ d \mid \big( e^{[M : \widetilde{M}]} \cdot d_{\gamma}^{[M : M_{\gamma}]} \big),$$ where $e$ is the relative discriminant of $\widetilde{M}$. By inductive assumption, we have that $$e \mid \prod_{i=1}^{\gamma - 1} d_{i}^{[\widetilde{M} : M_{i}]},$$ which upon using the fact that $[M : M_{i}] = [M : \widetilde{M}] \cdot [\widetilde{M} : M_{i}]$ proves the proposition. 
\end{proof}

\begin{prop}
When $K$ is a number field, Using the notations defined for Theorem $1$, the absolute discriminant $d_{L}$ divides $N_{K/\mathbb{Q}}(D)$.
\end{prop}

\begin{proof}
It suffices to prove that relative discriminant $d_{L/K}$ of $L$ divides $D$ in $K$; then the result follows by taking norm $N_{K/\mathbb{Q}}$ to the ensuing divisibility relation. For every $i \in\{1, 2, \ldots, m\}$: \vspace{2mm}
\begin{itemize}
    \item Let $\alpha_{i}^{(1)} = \alpha_{i}, \alpha_{i}^{(2)}, \ldots, \alpha_{i}^{(n_{i})}$ be the $K$-conjugates of $\alpha_{i}$. Here $\alpha_{i}$ is the fixed root of the irreducible factor $g_{i}$ of $f$. \vspace{2mm}
    
    \item For every $j \in\{1, 2, \ldots, n_{i}\}$, let $L_{i}^{(j)}$ be the field $K(\alpha_{i}^{(j)})$. Since each of the $\alpha_{i}^{(j)}$ are the roots of the same irreducible polynomial $g_{i}$, we have that the relative discriminant of each of the $L_{i}^{(j)}$ is equal to the relative discriminant $d_{i}$ of $L_{i}$.
    
    \item Let $\widetilde{L_{i}} = K(\alpha_{i}^{(1)}, \alpha_{i}^{(2)}, \ldots, \alpha_{i}^{(n_{i}-1)}, \alpha_{i}^{(n_{i})})$ be the splitting field of $g_{i}$, which is the composite of the fields $L_{i}^{(1)}, L_{i}^{(2)}, \ldots$ , $L_{i}^{(n_{i})}$.
\end{itemize}

Therefore, the relative discriminant $\widetilde{d_{i}}$ of $\widetilde{L_{i}}$ is such that
\begin{equation}
    \widetilde{d_{i}} \mid \prod_{j=1}^{n_{i}} d_{i}^{[\widetilde{L_{i}} : L_{i}^{(j)}]} = d_{i}^{(n_{i}-1)[\widetilde{L_{i}} : L_{i}^{(j)}]} \mid d_{i}^{(n_{i}-1)(n_{i}-1)!} \mid \mathfrak{q}_{i}^{(n_{i}-1)(n_{i}-1)!} \label{divisibility1}
\end{equation}
The first divisibility is from Proposition $2$ and the last divisibility relation in \eqref{divisibility1} is a consequence of the fact that the discriminant $\mathfrak{q}_{i}$ of the polynomial $g_{i}$ is always divisible by the relative discriminant $d_{i}$ of $L_{i} = K(\alpha_{i})$, which we explain below. 

For a fixed $i$, let $n = n_{i}$ which is the degree of the irreducible factor $g_{i}$, $a_{j} = a_{ij}$ and $\alpha = \alpha_{i}$. Consider the basis $\{\beta_{1}, \beta_{2}, \ldots, \beta_{n}\}$ of $L_{i}$ over $K$:
\begin{multline*}
    \\ \beta_{1} = 1 \\ \beta_{2} = a_{n}\alpha \\ \ldots \\ \beta_{i} = a_{n}\alpha^{i-1} + a_{n-1}\alpha^{i-2} + \ldots + a_{n-i+1}\alpha \\ \ldots \\ \beta_{n} = a_{n}\alpha^{n-1} + a_{n-1}\alpha^{n-2} + \ldots + a_{1}\alpha \\
\end{multline*}
We claim that $\beta_{j}$ are integral for every $j = 1, 2, \ldots, n$. Note that we have \[\beta_{n} a_{n}\alpha^{n-1} + a_{n-1}\alpha^{n-2} + \ldots + a_{n}\alpha = \frac{-a_{0}}{\alpha} - a_{1} \] and for every $1 \leq i \leq (n-2)$ one has
\begin{multline}
   \\ \beta_{n-i} = a_{n} \alpha^{n-i-1} a_{n-1} \alpha^{n-i-2} + \ldots + a_{i+2}\alpha \\ = - \frac{a_{0}}{\alpha^{i+1}} - \frac{a_{1}}{\alpha} - \ldots  - \frac{a_{i}}{\alpha} - a_{i+1} \\ = \frac{\beta_{n-i+1}}{\alpha} - a_{i+1} \\ \label{mid}
\end{multline}
Therefore, we will show that $\beta_{n-i}$ is integral for every $0 \leq i \leq (n-1)$ by induction. For $i = 0$, note that since $\beta_{n} = \frac{-a_{0}}{\alpha} - a_{1}$ it suffices to prove that $y := \frac{a_{0}}{\alpha}$ is integral. Consider the following series of equivalent statements.  
\begin{multline*}
    \\ a_{n}\alpha^{n} + a_{n-1} \alpha^{n-1} + \ldots + a_{1} \alpha + a_{0} = 0 \\ \Leftrightarrow a_{n}\alpha^{n-1} + a_{n-1} \alpha^{n-2} + \ldots + a_{1} + \frac{a_{0}}{\alpha} = 0 \\ \Leftrightarrow \frac{a_{n}a_{0}^{n-1}}{y^{n-1}} + \frac{a_{n-1}a_{0}^{n-2}}{y^{n-2}} + \ldots + \frac{a_{2}a_{0}}{y}+a_{1}+y = 0 \\ \Leftrightarrow y^{n} + a_{1}y^{n-1} + a_{2}a_{0}y^{n-2} + \ldots a_{n-1}a_{0}^{n-2}y + a_{n}a_{0}^{n-1} = 0,\\
\end{multline*}
which shows that $y$ and hence $\beta_{n}$ is integral. Now, we assume that $\beta_{n-i+1}$ is integral and show that $z := \frac{\beta_{n-i+1}}{\alpha}$ is also integral. Again, consider the following series of equivalent statements. 
\begin{multline*}
    \\ a_{n}\alpha^{n} + a_{n-1} \alpha^{n-1} + \ldots + a_{1} \alpha + a_{0} = 0 \\\vspace{1mm} \Leftrightarrow a_{n}\alpha^{n- i} + a_{n-1} \alpha^{n- i -1} + \ldots + a_{i+1}\alpha + \Big( a_{i} + \frac{a_{i-1}}{\alpha} + \ldots + \frac{a_{1}}{\alpha^{i-1}} + \frac{a_{0}}{\alpha^{i}} \Big) = 0 \\\vspace{1mm} \Leftrightarrow a_{n}\frac{(\beta_{n-i+1})^{n-i}}{z^{n-i}} + a_{n-1}\frac{(\beta_{n-i+1})^{n-i-1}}{z^{n-i-1}} + \ldots + a_{i+1}\frac{(\beta_{n-i+1})}{z} - \beta_{n-i+1} = 0 \\\vspace{1mm} \Leftrightarrow a_{n}\frac{(\beta_{n-i+1})^{n-i-1}}{z^{n-i}} + a_{n-1}\frac{(\beta_{n-i+1})^{n-i-2}}{z^{n-i-1}} + \ldots + a_{i+1}\frac{1}{z} - 1 = 0 \\\vspace{1mm} \Leftrightarrow
    -z^{n-i} + a_{i+1}z^{n-i-1} + a_{i+2} (\beta_{n-i+1}) z^{n-i-2} + \ldots + a_{n-1} \big( \beta_{n-i+1} \big)^{n-i-2} z +  a_{n} \big( \beta_{n-i+1} \big)^{n-i-1} = 0,\\
\end{multline*}
which shows that $z = \frac{\beta_{n-i+1}}{\alpha}$ and hence $\beta_{n-i} = z - a_{i+1} $ is also integral. Since $\beta_{1}, \beta_{2}, \ldots, \beta_{n}$ are integral in $L_{i}$, we have that $\mathcal{O}[\beta_{1}, \beta_{2}, \ldots, \beta_{n}] \subseteq \mathcal{O}_{L_{i}}$. However, we note that $d(\beta_{1}, \beta_{2}, \ldots, \beta_{n}) = \Big| \text{det} [\beta_{ij}] \Big| = \mathfrak{q}_{i}$, where $\beta_{ij}$ is defined exactly similarly as $\beta_{i}$ but with the $K$-conjugate $\alpha^{(j)}$ of $\alpha$.Therefore, from a well-known relationship between the discriminants of $\mathcal{O}_{K}$-submodules (for instance, see \cite[pp. 11, Proposition 2.12]{Neukrich}) we have that
\[d(\beta_{1}, \beta_{2}, \ldots, \beta_{n}) = \mathfrak{q}_{i} = \Big[ \mathcal{O}_{L_{i}} : \mathcal{O}[\beta_{1}, \ldots, \beta_{n}] \Big]^{2} d_{i},\] 
i.e, $d_{i} \mid \mathfrak{q}_{i}$. 

On the other hand, the splitting field $L$ of the polynomial $f$ is the composite of the splitting fields $\widetilde{L_{1}}, \widetilde{L_{2}}, \ldots, \widetilde{L_{m}}$ of the irreducible factors $g_{1}, g_{2}, \ldots, g_{m}$ respectively. Hence if $d_{L/K}$ is the relative discriminant of $L$ then, Proposition $4$ gives
\begin{equation}
    d_{L/K} \mid \prod_{i=1}^{m} ( \widetilde{d_{i}})^{[L : \widetilde{L_{i}}]} \mid \prod_{i=1}^{m} \bigg( \big( \widetilde{d_{i}} \big)^{n_{1}! n_{2}!\ldots n_{i-1}!n_{i+1}!\ldots n_{m}!}\bigg). \label{divisibility2}
\end{equation}
From \eqref{divisibility1} and \eqref{divisibility2}, we have that $$d_{L/K} \mid \prod_{i=1}^{m} \big( \mathfrak{q}_{i} \big)^{(n_{1}!\cdots n_{m}!)(1 - \frac{1}{n_{i}})} = D.$$
\end{proof}
\subsection{Least Prime Ideal in the Chebotarev Density Theorem}
To show that the condition $(3)$ in Theorem $1$ is equivalent to the condition $(2)$, we will employ upper bounds on the prime ideal in the Chebotarev density theorem. The upper bound is well known in the number field case and we use the following version that is taken from \cite{AhKw}. 
\begin{prop}\label{leastprime1}
For every finite extension $K$ of $\mathbb{Q}$, for every finite Galois extension $L (\neq\mathbb{Q})$ of $K$ and every conjugacy class $\mathcal{C}$ of Gal$(L/K)$, there exists a prime ideal $\mathfrak{p}$ of $K$ that is unramified in $L$ such that:
\begin{itemize}
    \item $\bigg[ \frac{L/K}{\mathfrak{p}} \bigg] = \mathcal{C}$ and\vspace{2mm}
    
    \item  $N_{K/\mathbb{Q}}(\mathfrak{p})$ is a rational prime which satisfies \[N_{K/\mathbb{Q}}(\mathfrak{p}) \leq  |d_{L}|^{12577} .\]
\end{itemize}
\end{prop}
We also need a similar upper bound on the degree of the least prime ideal in the Chebotarev density theorem, when $K$ is a finite separable extension of $\mathbb{F}_{q}(T)$. Such an upper bound should ideally depend only upon the field $K$ and the given polynomial $f$. Assume that $\mathbb{F}_{q^{l}}$ is the integral closure of $\mathbb{F}_{q}$ in $L$, for some $l \geq 1$. The map $$\varphi : \mathbb{F}_{q^{l}} \longrightarrow \mathbb{F}_{q^{l}} \text{ given by } \varphi(x) = x^{q}$$ is the generator of the Gal$(\mathbb{F}_{q^{l}}/\mathbb{F}_{q})$. On the other hand, we also have the surjective restriction Gal$(L/K) \rightarrow \hspace{-8pt} \rightarrow$ Gal$(\mathbb{F}_{q^{l}}/\mathbb{F}_{q})$. Since Gal$(\mathbb{F}_{q^{l}}/\mathbb{F}_{q})$ is abelian, the image of a conjugacy class $\mathcal{C}\subset$ Gal$(L/K)$ is a singleton set in Gal$(\mathbb{F}_{q^{l}}/\mathbb{F}_{q})$. Let $\varphi_{\mathcal{C}} := \varphi^{n} \in $Gal$(\mathbb{F}_{q^{l}}/\mathbb{F}_{q})$ be the image of every element of $\mathcal{C}$ under this restriction map. Given a conjugacy class $\mathcal{C}$ in $G = $ Gal$(L/K)$, define $$\pi_{\mathcal{C},n} := \{ \mathfrak{p} \in\mathbb{P}_{n}(K) : \bigg[\frac{L/K}{\mathfrak{p}}\bigg] = \mathcal{C} \}.$$ In addition, define $C_{1} = [L : K\mathbb{F}_{q}^{l}]$. We will use the following effective version of the Chebotarev density theorem in function fields (see \cite[Proposition 6.4.8]{FrJar}).
\begin{prop}
Let $\mathcal{C}$ be a conjugacy class in $G = \text{Gal }(L/K)$ and $n \in\mathbb{N}$. If $\varphi_{\mathcal{C}} = \varphi^{n}$, then \begin{equation}
    \bigg\lvert\pi_{\mathcal{C}, n} - \frac{|\mathcal{C}|}{C_{1}} \frac{q^{n}}{n}\bigg\rvert < \frac{2|\mathcal{C}|}{C_{1} \cdot n} \Big[(C_{1} + g_{L})q^{n/2} + C_{1}(2g_{K}+1) q^{n/4} + g_{L} + C_{1} d \Big]. \label{CDTFF}
\end{equation}
Here, $d = [K : \mathbb{F}_{q}(T)]$. On the other hand, if $\varphi_{\mathcal{C}} \neq \varphi^{n}$ then $\pi_{\mathcal{C}, n}$ is empty. 
\end{prop}
Given a Galois extension $L$ of $K$, which is a finite separable extension of $\mathbb{F}_{q}(T)$, we have $K \subseteq K\mathbb{F}_{q^{l}} \subseteq L$. Here, $K\mathbb{F}_{q^{l}}/K$ is a constant field extension and $L/K\mathbb{F}_{q^{l}}$ is a geometric field extension. To obtain the upper bound on the least prime ideal, we will use the Riemann-Hurwitz theorem (see \cite[Theorem 7.16]{Rosen}). 
\begin{prop}
Let $E_{1}/E_{2}$ be a finite, separable and geometric extension of functions fields. Then, 
\begin{equation}
    2g_{E_{2}}-2 = [E_{2}:E_{1}] (2g_{E_{1}} -2 ) + \text{deg}_{E_{2}}(D_{E_{2}/E_{1}}), \label{RH}
\end{equation}
where $D_{E_{2}/E_{1}}$ is the different divisor of $E_{2}/E_{1}$.
\end{prop}
Now, we are ready to prove an upper bound on the least prime ideal in the Chebotarev density theorem for function fields. 
\begin{prop}\label{leastprime2}
Let $q$ be a prime power, $K$ be a finite separable extension of $\mathbb{F}_{q}(T)$. Using notations defined in Theorem $1$, for any conjugacy class $\mathcal{C} \subset \text{Gal }(L/K)$ with $\varphi_{\mathcal{C}} = \varphi^{n}$, there exists an unramified prime $\mathfrak{p}$ in $K$ such that
\begin{itemize}
    \item $\bigg[ \frac{L/K}{\mathfrak{p}} \bigg] = \mathcal{C}$ and
    
    \item $\text{deg}(\mathfrak{p}) \leq \lceil 2 \log_{q}\big( 2\Delta^{\prime} + 2dD^{\prime} + 8g_{K}D^{\prime} + 4 \big) \rceil$.
\end{itemize}
\end{prop}

\begin{proof}
Equation \eqref{CDTFF} implies that $\pi_{\mathcal{C},n}$ is positive as soon as $\frac{|\mathcal{C}|}{C_{1}} \frac{q^{n}}{n}$ is greater than or equal to the RHS of \eqref{CDTFF}. Note that the RHS of \eqref{CDTFF} is smaller than 
\begin{equation*}
    \frac{2|\mathcal{C}|}{C_{1} \cdot n} \bigg(2C_{1} + 2g_{L} + 2C_{1}g_{K} + dC_{1} \bigg) q^{n/2}.
\end{equation*}
Therefore, an upper bound on the smallest $n$ such that $\pi_{\mathcal{C},n} > 0$ is the least $n$ such that 
\begin{equation*}
    \frac{|\mathcal{C}|}{C_{1}} \frac{q^{n}}{n} \geq \frac{2|\mathcal{C}|}{C_{1} \cdot n} \bigg(2C_{1} + 2g_{L} + 2C_{1}g_{K} + dC_{1} \bigg) q^{n/2},
\end{equation*}
i.e.
\begin{equation}
    n = \lceil 2  \log_{q}\bigg(4C_{1} + 4g_{L} + 4C_{1}g_{K} + 2dC_{1} \bigg) \rceil. \label{bound1}
\end{equation}
Consider the tower of field extensions $K \subseteq K\mathbb{F}_{q^{l}} \subseteq L$. Since genus does not change in a constant field extension over a perfect field (for instance, using \cite[Proposition 8.9]{Rosen}), we have that the genus of $K\mathbb{F}_{q^{l}}$ is equal to genus $g_{K}$ of $K$. On the other hand, using the Riemann-Hurwitz formula \eqref{RH} for the geometric extension $L/K\mathbb{F}_{q^{l}}$ with $E_{2} = L$ and $E_{1} = K\mathbb{F}_{q^{l}}$ gives us $$2g_{L} - 2 = (2g_{K} - 2) C_{1} + \text{deg}_{E_{2}}(D_{E_{2}/E_{1}}),$$ and hence
\begin{equation}
     4g_{L} = 4g_{K}C_{1} - 4C_{1} + 4 + 2 \cdot \text{deg}_{E_{2}}(D_{E_{2}/E_{1}}).\label{RH1}
\end{equation}
The degree of the different divisor is $ \sum_{\mathfrak{p}\in\mathbb{P}(K)} \delta(\mathfrak{p}) \text{ deg}(\mathfrak{p})$, where $\delta(\mathfrak{p})$ is the highest power of $\mathfrak{p}$ dividing the discriminant and the sum is taken over the primes $\mathfrak{p}$ that are ramified. Since the relative discriminant $\mathfrak{q} = \mathfrak{p}_{1}^{b_{1}} \cdots \mathfrak{p}_{\nu}^{b_{\nu}}$ is the relative norm of the different, the set of primes of $\mathbb{F}_{q}[T]$ that ramify in $L$ are a subset of $\{\mathfrak{p_{1}}, \ldots, \mathfrak{p_{\nu}}\}$. Therefore,
\begin{equation}
    \text{deg}_{E_{2}}(D_{E_{2}/E_{1}}) \leq \sum_{i=1}^{\nu} b_{i} \text{ deg}(\mathfrak{p}_{i}). \label{different}
\end{equation}
Using \eqref{RH1} and \eqref{different}, we have that, 
\begin{equation*}
    4g_{L} \leq 4g_{K}C_{1} -4C_{1} + 4 + 2 \cdot \sum_{i=1}^{\nu} b_{i} \text{ deg}(\mathfrak{p}_{i}) = 4g_{K}C_{1} - 4C_{1} + 4 + 2\Delta^{\prime}, 
\end{equation*}
which implies that 
\begin{equation}
    (4C_{1} + 4g_{L} + 4C_{1}g_{K} + 2dC_{1}) \leq 8g_{K}C_{1} + 2dC_{1} + 4 + 2\Delta^{\prime}. \label{bound2}
\end{equation}
Since the splitting field of the polynomial $f$ is the compositum of the splitting fields of its individual irreducible factors, we have that $[L : K] \leq \prod_{i=1}^{\nu} n_{i}!$. Therefore,
\begin{equation}
    C_{1} = [L : K\mathbb{F}_{q^{l}}] \leq [L : K] \leq \prod_{i=1}^{\nu} n_{i}! = D^{\prime} \label{C1}
\end{equation} 
Using \eqref{bound1}, \eqref{bound2} and \eqref{C1}, we have that an upper bound for the degree $n$ of the prime ideal in the Chebotarev density theorem is given by $$\lceil 2 \log_{q} \big( 2\Delta^{\prime} + 2dD^{\prime} + 8g_{K}D^{\prime} + 4 \big) \rceil$$
which proves the proposition.
\end{proof}

\section{Proof of Theorem 1}
We will show that $(2) \longrightarrow (1) \longrightarrow (3) \longrightarrow (2)$, among which the implication $(1) \longrightarrow (3)$ follows trivially from definition of an intersective polynomial.

\textbf{Proof of (2) implies (1).} 

If a prime $\mathfrak{p}\in\mathcal{O}_{K}$ does not divide $\mathfrak{q}$ then $f$ has a root in $K_{\mathfrak{p}}$, as a consequence of Proposition $1$. To see this, we note that \[\text{Gal} (L / K) = G = \bigcup_{\sigma\in G} \sigma^{-1} U \sigma \] implies that every conjugacy class of $G$ intersects $U = \bigcup_{i=1}^{m} H_{i}$. Using the notation $H_{i} =$ Gal $(L/K(\alpha_{i}))$ defined for Theorem $1$, we have that every conjugacy class $\bigg[ \frac{L/K}{\mathfrak{p}} \bigg]$ intersects Gal $(L/K(\alpha_{j}))$ for some $j \in\{1, 2, \ldots, m\}$. In other words, some element of $\bigg[ \frac{L/K}{\mathfrak{p}} \bigg]$ fixes root of some irreducible factor $g_{j}$ of $f$. Therefore, Proposition $1$ implies that $g_{j}$, and hence $f$, has a root modulo $\mathfrak{p}$. 

If $\mathfrak{p}$ divides $\mathfrak{q}$ (i.e. $\mathfrak{p} = \mathfrak{p}_{i}$) then, pick a $\gamma \in\mathcal{O}_{K}$ such that $f(\gamma) \equiv 0 \hspace{1mm} (\text{mod } \Delta)$. Since $\mathfrak{p} = \mathfrak{p}_{i}$ we have $|\mathfrak{q}|_{\mathfrak{p}} = \mathfrak{p}_{i}^{-b_{i}}$. Furthermore, since $|\Delta|_{\mathfrak{p}} = \mathfrak{p}_{i}^{-2 b_{i} - 1}$, $$|f(\gamma)|_{\mathfrak{p}} \leq \mathfrak{p}_{i}^{-2b_{i}-1} < \mathfrak{p}^{-2b_{i}} = |\mathfrak{q}^{2}|_{\mathfrak{p}}.$$ Since $|f(\gamma)|_{\mathfrak{p}} < |\mathfrak{q}^{2}|_{\mathfrak{p}}$, we have that $|g_{j}(\gamma)|_{\mathfrak{p}} < |\mathfrak{r}_{j}|^{2}_{\mathfrak{p}}$ for some $j \in\{1, 2, \ldots, m\}$. On the other hand, by definition of the resultant $\mathfrak{r}_{j}$, there exists $A(x), B(x) \in\mathcal{O}_{K}[x]$ such that $$A(x)g_{j}(x) + B(x)g_{j}^{\prime}(x) = \mathfrak{r}_{j}.$$ Therefore, we have
\begin{multline*}
    |g_{j}^{\prime}(\gamma)|_{\mathfrak{p}} = |g_{j}^{\prime}(\gamma)B(\gamma)|_{\mathfrak{p}} = |\mathfrak{r}_{j} - A(\gamma)g_{j}(\gamma)|_{\mathfrak{p}} \\ = \text{max } \{|\mathfrak{r}_{j}|_{\mathfrak{p}}, |A(\gamma)g_{j}(\gamma)|_{\mathfrak{p}}\} = \text{max } \{|\mathfrak{r}_{j}|_{\mathfrak{p}}, |g_{j}(\gamma)|_{\mathfrak{p}}\} = |\mathfrak{r}_{j}|_{\mathfrak{p}}.
\end{multline*}
We have used above that $|A(\gamma)|_{\mathfrak{p}} = |B(\gamma)|_{\mathfrak{p}} = 1$ because $A(x), B(x) \in\mathcal{O}_{K}[x]$ and $\gamma\in\mathcal{O}_{K}$ and also the fact that $g_{j}(x)$ is a separable polynomial and hence its derivative is not identically zero. Therefore, $|g_{j}^{\prime}(\gamma)|_{\mathfrak{p}}^{2} = |\mathfrak{r}_{j}|_{\mathfrak{p}}^{2} > |g_{j}(\gamma)|_{\mathfrak{p}}$. Then, $g_{j}$ has a root in $K_{\mathfrak{p}}$ by the Hensel's lemma (for example, see \cite[Proposition 3.5.2]{FrJar}). Now the result follows from the standard application of the Chinese Remainder Theorem. $\square$

\textbf{Proof of (3) implies (2).} 
For primes $\mathfrak{p}$ dividing $\mathfrak{q}$, start by picking a $\gamma\in\mathcal{O}_{K}$ such that $f(\gamma) \equiv 0 \hspace{1mm} (\text{mod } \Delta)$. We can show exactly similarly as in the proof of $(2)$ implies $(1)$ that $f$ has a root in $K_{\mathfrak{p}}$. For the primes $\mathfrak{p}$ that do not divide $\mathfrak{q}$, lets prove the result in the number field case first. We first assume that $L \neq \mathbb{Q}$. Proposition $4$ implies that there exists an unramified prime $\mathfrak{p}^{\prime}$  such that 
\begin{enumerate}[(a)]
    \item $\big[\frac{L/K}{\mathfrak{p}}\big] = \big[\frac{L/K}{\mathfrak{p}^{\prime}}\big]$ and
    
    \item $N_{K/\mathbb{Q}}(\mathfrak{p}^{\prime}) \leq  |d_{L}|^{12577} $
\end{enumerate}
as long as $L \neq \mathbb{Q}$. First note that the condition $(b)$ above can be replaced by 
\begin{equation}
N_{K/\mathbb{Q}}(\mathfrak{p}^{\prime}) \leq  |N_{K/\mathbb{Q}}(D)|^{12577}, \label{least number field}
\end{equation}
because $d_{L}$ divides $N_{K/\mathbb{Q}}(D)$ as proved in Proposition $3$. Both $(a)$ and $(b)$ together imply that 
$$\big[\frac{L/K}{\mathfrak{p}}\big] \cap U \neq \emptyset$$ if and only if $$\big[\frac{L/K}{\mathfrak{p}^{\prime}}\big] \cap U \neq \emptyset$$ for $\mathfrak{p}^{\prime}$ satisfying \eqref{least number field}. In other words, $f$ has a root in $K_{\mathfrak{p}}$ if and only if $f$ has a root in $K_{\mathfrak{p}^{\prime}}$ satisfying \eqref{least number field}. By assumption of $(3)$, along with Proposition $3$, we have that the set of $\big[\frac{L/K}{\mathfrak{p}^{\prime}}\big]$ as $\mathfrak{p}^{\prime}$ runs through \[N_{K/\mathbb{Q}}(\mathfrak{p}^{\prime}) \leq  |N_{K/\mathbb{Q}}(D)|^{12577}  \]  is all of $G$. In other words, every conjugacy class in $G$ intersects $U$ and hence $\cup_{\sigma \in G} \big( \sigma^{-1} U \sigma \big) = G$.
When $L = \mathbb{Q}$, $f$ splits into linear factors and all the roots of $f$ are rational. Therefore, the Galois group only contains identity, which fixes the root of every linear polynomial. Therefore, the Galois-theoretic condition in $(2)$ is trivially satisfied. 

In the case when $K$ is an extension of $\mathbb{F}_{q}(T)$, the proof follows from exactly the same argument; except using the bound on the degree of least prime ideal as given in Proposition \ref{leastprime2}. $\square$

\section{Corollaries, Examples and Discussion.}
\begin{corollary}
Let $f(x) \in \mathcal{O}_{K}[x]$ be a polynomial with separable irreducible factors. If $f$ is irreducible and of degree $\geq 2$, then $f$ cannot be intersective. On the other hand, if $f$ is intersective but does not have a root in $K$, then degree of $f$ must be at least $5$.
\end{corollary}

\begin{proof}
We will use the notation defined before Theorem $1$. Assume that $f$ is irreducible and of degree $\geq 2$. From irreducibility of $f$, we have $U = H_{1}$ and since $\text{deg }(f) > 1$, we also have $G \neq U$. Since $U = H_{1}$ is a proper subgroup of $G$, we have that $G/U = \{g_{1}U, g_{2}U, \ldots, g_{k}U\}$ for $k = \frac{|G|}{|U|} \geq 2$. Therefore we have, 
\[ \big| \bigcup_{\sigma\in G} \sigma^{-1} U \sigma \big| = \big| \bigcup_{i=1}^{k} g_{i}^{-1} U g_{i} \big| < k |U| - (k-1) = |G| - (k-1) < |G|\] because $k \geq 2$. Therefore, we have $G \neq \bigcup_{g \in G} g^{-1} U g$ and hence Theorem $1$ implies that $f$ cannot have a root modulo every ideal of $\mathcal{O}_{K}$.

Now, assume that $\text{deg }(f) \leq 4$ and that $f$ has no roots in $K$. We will prove that $f$ cannot have root modulo every ideal in $\mathcal{O}_{K}$. Since $f$ satisfies \eqref{defn1}, $f$ must be reducible by Corollary $1$. The reducibility of $f$, along with the fact that $f$ has no roots in $K$, implies that \[f(x) = \begin{cases} g_{1}(x) g_{2}(x) \hspace{2mm}\text{ or } \\ g_{1}(x)^{2} \hspace{2mm}\text{ or } \\ g_{1}(x) \end{cases},\] where $g_{1}$ and $g_{2}$ are irreducible quadratic polynomials in $\mathcal{O}_{K}[x]$. In the first case, the splitting field $L$ of $f$ has degree $4$ over $K$, whereas the splitting field of $L$ has degree $2$ in the latter two cases. 

We first deal with the case when the splitting field has degree $4$. Let $\alpha_{i}, \beta_{i}$ be roots of $g_{i}$, for $i = 1, 2$. Then the splitting field of $f$ is $L = K(\alpha_{1}, \alpha_{2})$ and the Galois group  of splitting field of $f$ is $G = \{e, \sigma_{1}, \sigma_{2}, \sigma_{1}\sigma_{2} \}$. Here, $e$ is the identity map and
\begin{equation*}
    \sigma_{i}(\alpha_{i}) = \beta_{i} \hspace{2mm} \text{ and } \sigma_{i}(\beta_{i}) = \alpha_{i} \hspace{2mm} \text{ and $\sigma_{i}$ fixes the roots of the other quadratic factor}.
\end{equation*}
In other words, $G$ is isomorphic to the Klein $4$-group. Then, in the notation of Theorem $1$ we have, 
\begin{equation*}
    H_{1} = \text{ Gal}(L/K(\alpha_{1})) =\{e, \sigma_{2} \} \hspace{2mm} \text{ and } H_{2} = \text{ Gal}(L/K(\alpha_{2})) = \{e, \sigma_{1} \}.
\end{equation*}
Therefore, $U = H_{1} \cup H_{2} = \{e, \sigma_{1}, \sigma_{2} \}$ and hence $\cup_{g \in G} g^{-1} U g = U \neq G$ i.e. $f$ does not satisfy condition $(2)$ of Theorem $1$. Therefore, $f$ cannot be intersective. 

Now we deal with the case when splitting field $L$ of $f$ has degree $2$ over $K$. In this case, let $\alpha$ be the root of only irreducible quadratic factor $g_{1}$. The splitting field is $L = K (\alpha)$ with $[L : K] = 2$ and the Galois group $G = $ Gal $(L/K) \simeq \mathbb{Z}/2\mathbb{Z}$ is generated by $\sigma (\alpha) = - \alpha$. In the notation of Theorem $1$, the only element of Gal $(L/K)$ that fixes the root $\alpha$ is the identity element, i.e., $U = \{e\}$. Therefore, 
\[\{e\} = \bigcup_{\sigma\in G} \big(\sigma^{-1} U \sigma\big) \neq G\] and hence $f$ cannot be intersective.
\end{proof}

\subsection{An Example.} \label{Example} In this subsection, we will use Theorem $1$ to construct a family of examples of intersective polynomials.  Let $q$ be a power of some odd prime, $K = \mathbb{F}_{q}(T)$ and let $$f(x) = (x^{2} - \theta_{1}) (x^{2} - \theta_{2}) (x^{2} - \theta_{1}\theta_{2}) \in\mathcal{O}_{K}[x],$$ where $\theta_{1}$ and $\theta_{2}$ are primes in $\mathcal{O}_{K} = \mathbb{F}_{q}[T]$ that are non-associates. Recall that we require $q$ to be odd here to ensure that the irreducible factors of $f$ are all separable. To use the notations defined for Theorem $1$ we denote the irreducible factors of $f$ as:
\begin{equation*}
    g_{1}(x) = (x^{2} - \theta_{1}), \hspace{2mm} g_{2}(x) = (x^{2} - \theta_{2}), \hspace{2mm} g_{3}(x) = (x^{2} - \theta_{1}\theta_{2}).
\end{equation*}
Since none of the $\theta_{i}$ can be square in $\mathcal{O}_{K}$, the splitting field of the polynomial $f$ is $$L = K(\sqrt{\theta_{1}}, \sqrt{\theta_{2}}).$$ Fix the roots $\alpha_{1} = \sqrt{\theta_{1}}$,  $\alpha_{2} = \sqrt{\theta_{2}}$ , $\alpha_{3} = \sqrt{\theta_{1}\theta_{2}}$ of $g_{1}, g_{2}$ and $g_{3}$ respectively. Note that the Galois group of this extension is given by $$ G  = \{ e, \sigma_{1}, \sigma_{2}, \sigma_{1}\sigma_{2} \}.$$ Here $e$ is the identity map and for every $i, j \in\{1, 2\}$
\begin{equation*}
    \sigma_{i}(\sqrt{\theta_{j}}) = \begin{cases} -\sqrt{\theta_{j}} \hspace{2mm}; i = j \\ \sqrt{\theta_{j}} \hspace{2mm}; i \neq j. \end{cases} 
\end{equation*}
So, $G$ is isomorphic to the Klein four-group, where each non-identity element has order $2$. As in Theorem $1$, let $H_{i}$ be the subgroup of $G$ that fixes the root $\alpha_{i}$ of $g_{i}$. Then, 
\begin{itemize}
    \item $H_{1} = \text{ Gal}(L/K(\sqrt{\alpha_{1}}) = \{e, \sigma_{2}\} \leq G$,
    
    \item $H_{2} = \text{ Gal}(L/K(\sqrt{\alpha_{2}}) = \{e, \sigma_{1}\} \leq G$,
    
    \item $H_{3} = \text{ Gal}(L/K(\sqrt{\alpha_{3}}) = \{e, \sigma_{1}\sigma_{2} \} \leq G$.
\end{itemize}
Therefore, $$\bigcup_{\sigma\in G} \big( \sigma^{-1}U\sigma \big) = U = H_{1} \cup H_{2} \cup H_{3} \cup H_{4} = \{ e, \sigma_{1}, \sigma_{2}, \sigma_{1}\sigma_{2} \} =  G.$$ 

In addition, we also have that $\mathfrak{q} = \mathfrak{q}_{1} \cdot \mathfrak{q}_{2} \cdot \mathfrak{q}_{3}$ where $\mathfrak{q}_{1}, \mathfrak{q}_{2} \text{ and } \mathfrak{q}_{3}$ are ideals generated by $\theta_{1}, \theta_{2}$ and $\theta_{1}\theta_{2}$ in $\mathbb{F}_{q}[T]$ (since $q$ is odd). Therefore, $\mathfrak{q} = \langle \theta_{1}^{2}, \theta_{2}^{2} \rangle$ and hence $\Delta = \langle \theta_{1}^{5}, \theta_{2}^{5} \rangle$. Since $\theta_{1}$ and $\theta_{2}$ are non-associates, using the Chinese Remainder Theorem and $(2)$ of Theorem $1$, we have that $f(x)$ is intersective if and only if $f(x) \equiv 0 \hspace{1mm} (\text{mod } \langle \theta_{i}^{5} \rangle)$ is solvable for $i = 1, 2$. 

In fact, a similar argument also gives the same result when the above polynomial $f(x)$ has coefficients over the ring of integers of a number field. Hence, we have the following corollary, where $\langle x \rangle$ is used to denote the ideal $x \mathcal{O}_{K}$.
\begin{corollary}
Let $K$ either be a number field or $\mathbb{F}_{q}(T)$ for an integer $q$ that is a power of an odd prime. Let $\theta_{1}$ and $\theta_{2}$ be non-associate primes in $\mathcal{O}_{K}$. Then the polynomial $$f(x) = (x^{2} - \theta_{1}) (x^{2} - \theta_{2}) (x^{2} - \theta_{1}\theta_{2}) \in \mathcal{O}_{K}[x]$$ is intersective if and only if $\theta_{1}$ is a square modulo $\langle \theta_{2}^{5} \rangle$ and $\theta_{2}$ is a square modulo $\langle \theta_{1}^{5} \rangle$. 
\end{corollary}

\begin{proof}
From the explanation in \ref{Example}, we have that $f(x)$ is intersective if and only if $f(x) \equiv 0 \hspace{1mm} (\text{mod } \langle \theta_{i}^{5} \rangle)$ is solvable for $i = 1, 2$. It suffices to show that $f(x) \equiv 0 \hspace{1mm} (\text{mod } \langle \theta_{1}^{5} \rangle)$ is solvable if and only if $\theta_{2}$ is a square modulo $\langle \theta_{1}^{5} \rangle$. Analogously, the same fact will hold when the roles of $\theta_{1}$ and $\theta_{2}$ is reversed. 

Now, we will show that neither $\theta_{1}$ nor $\theta_{1}\theta_{2}$ can be a square modulo $\langle\theta_{1}^{5}\rangle$. In fact, we will only show that $\theta_{1}\theta_{2}$ cannot be a square modulo $\langle \theta_{1}^{5} \rangle$. An exactly analogous proof, with $\theta_{1}\theta_{2}$ replaced by $\theta_{1}$, that $\theta_{1}$ is not a square modulo $\langle \theta_{1}^{5} \rangle$. 

Assume that $y^{2} \equiv \theta_{1}\theta_{2} \hspace{1mm} (\text{mod } \langle \theta_{1}^{5} \rangle)$ for some $y \in\mathcal{O}_{K}$, i.e., $\ y^{2} - \theta_{1}\theta_{2} \in \langle \theta_{1}^{5} \rangle \subseteq \langle \theta_{1} \rangle$. Since, $y^{2} - \theta_{1}\theta_{2} \in \langle \theta_{1} \rangle$, we have that $ \langle \theta_{1} \rangle \mid \langle y^{2} - \theta_{1} \rangle$ which implies that $\langle \theta_{1} \rangle \mid \langle y^{2} \rangle$ and hence $\langle \theta_{1}^{2} \rangle  \mid \langle y^{2} \rangle$. However, $ \langle \theta_{1}^{2} \rangle \mid \langle y^{2} \rangle$ along with $ y^{2} - \theta_{1}\theta_{2} \in \langle \theta_{1}^{5} \rangle$ gives that $\langle \theta_{1}^{2} \rangle \mid \langle \theta_{1}\theta_{2} \rangle$. This is a contradiction because $\theta_{1}$ and $\theta_{2}$ are non-associates. Therefore, the corollary is established. 
\end{proof}

\subsection{Optimality of the Upper Bound in Theorem 1}
The Proposition $\ref{leastprime1}$, and hence upper bound in $(3)$ of Theorem $1$, can be improved in the case when $K$ is a number field, under the Generalized Riemann Hypothesis (GRH) for the Dedekind zeta function \[ \zeta_{L}(s) = \sum_{I \in\mathcal{O}_{L}} \frac{1}{N_{L/\mathbb{Q}}(I)}.\] For example, assuming validity of the GRH for $\zeta_{L}(s)$, the upper bound in Proposition $6$, can be replaced by $C (\log|d_{L}|)^{2}$ for some constant $C$ (see comments on Corollary $1.2$ in \cite[pp. 461--462]{LaO}). Furthermore, it was demonstrated in \cite{Fiori} that this is the best possible bound for a general $L$. Therefore, under GRH we can replace the upper bound in $(3)$ of Theorem $1$ by $$N_{K/\mathbb{Q}}(\mathfrak{p}) \leq C \big(\log|N_{K/\mathbb{Q}}(D)|\big)^{2}$$ in the number field case. The upper bound in Proposition \ref{leastprime1} admits significant asymptotic improvements (for large enough $d_{L}$) (see \cite{KaNW} and \cite{Zaman}). However, due to our interest in Theorem $1$ for a general polynomial, we cannot use those results.  

\section*{Acknowledgement} The author would like to thank Dr. Ofir Gorodetsky for an illuminating email conversation regarding the Chebotarev density theorem in global function fields.


\begin{thebibliography}{10}
\bibitem{BA} Ackelsberg, E. and Bergelson, V. (2021). A dynamical approach to bracket polynomials in finite characteristic and applications. \textit{Preprint}. \url{https://drive.google.com/file/d/1jqHeiYqQoShr_OiiDW6VUmbIJdMtlcd8/view}.

\bibitem{AhKw} Ahn, J-H. and Kwon, S-H. (2019). An explicit upper bound for the least prime ideal in the Chebotarev density theorem. \textit{Ann. Inst. Fourier (Grenoble)}. 69 (3). 1411--1458.

\bibitem{BerBil} Berend, D. and Bilu, Y. (1996). Polynomials with Roots Modulo Every Integer. \textit{Proc. Amer. Math. Soc.} 124 (6). 1663--1671.

\bibitem{BLL} Bergelson, V., Leibman, A. and Lesigne, E. (2008). Intersective Polynomials and Polynomial Szemeredi Theorem. \textit{Adv. Math.} 219. 369–388. 

\bibitem{BerRob} Bergelson, V. and Robertson, D. (2016). Polynomial Multiple Recurrence Over Rings of Integers. \textit{Ergodic Theory Dynam. Systems.} 36 (5). 1354--1378.

\bibitem{Fiori} Fiori, A. (2019). Lower Bounds for the Least Prime in Chebotarev. \textit{Algebra Number Theory}. 13 (9). 2199--2203.

\bibitem{FrJar} Fried, M. D. and Jarden, M. \textit{Field Arithmetic}, 2nd Ed. Ergenbnisse der Mathematik und ihrer Grenzgebiete. 3. Folge. A Series of Modern Surveys in Mathematics. vol. 11. Springer-Verlag, Berlin. 2005. 

\bibitem{KaNW} Kadiri, H., Ng, N. and Wong, P-J. (2019). The Least Prime Ideal in the Chebotarev Density Theorem. \textit{Proc. Amer. Math. Soc.} 147 (6). 2289--2303.

\bibitem{KaMF} Kamae, T. and Mend\'es-France, M. (1978). Van der Corput's Difference Theorem. \textit{Israel J. Math.} 31 (3-4). 335--342. 

\bibitem{LaO} Lagarias, J.C. and Odlyzko, A. M. (1977). Effective Versions of the Chebotarev Density Theorem. \textit{Algebraic Number Fields: L-functions and Galois properties (Proc. Sympos. Univ. Durham, Durham, 1975)}. Academic Press, London, 409--464.

\bibitem{Neukrich} Neukrich, J. \textit{Algebraic Number Theory}. (Translated from German by Norbert Schappacher). A Series of Comprehensive Studies in Mathematics Springer-Verlag, New York. 1999.

\bibitem{Rosen} Rosen, M. \textit{Number Theory in Function Fields}. Graduate Texts in Mathematics. Springer-Verlag, New York. 2002. 

\bibitem{Zaman} Zaman, A. (2017). Bounding the Least Prime Ideal in the Chebotarev Density Theorem. \textit{Funct. Approx. Comment. Math.} 57 (1). 115--142.


\end{thebibliography}
\end{document}